\documentclass[10pt,amstex]{article}

\usepackage{amsmath}
\usepackage{amscd}
\usepackage{fullpage}
\usepackage{graphicx}
\usepackage{caption}
\usepackage{setspace}
\usepackage{listings}
\usepackage{pgf,tikz}
\usetikzlibrary{shapes, arrows}
\usetikzlibrary{decorations.pathreplacing,patterns}
\usepackage{amsfonts}
\usepackage{amssymb}
\usepackage{tabularx}
\usepackage{enumerate}
\usepackage[title,titletoc,toc]{appendix}
\usepackage{amsthm}
\usepackage{subcaption}
\usepackage{color}
\usepackage[normalem]{ulem}

\DeclareMathAlphabet{\mathpzc}{OT1}{pzc}{m}{it}

\newcommand{\one}{^{(1)}}
\newcommand{\two}{^{(2)}}
\newcommand{\three}{^{(3)}}
\newcommand{\e}{\epsilon}

\newtheorem{thm}{Theorem}[section]
\newtheorem{lem}{Lemma}[section]
\newtheorem{prop}{Proposition}[section]
\newtheorem{cor}{Corollary}[section]
\newtheorem{definition}{Definition}[section]

\begin{document}

\title{Stability Conditions for Coupled Autonomous Vehicles Formations} 
\author{Pablo E. Baldivieso \thanks{Oregon State University - Cascades;
e-mail: pablo.baldivieso@osucascades.edu},
J. J. P. Veerman\thanks{Fariborz Maseeh Dept. of Math. and Stat., Portland State Univ.;
e-mail: veerman@pdx.edu}
}\maketitle

\doublespace
\vskip .0in

\noindent

\begin{abstract}
In this paper, we give necessary conditions for stability of coupled autonomous vehicles in $\mathbb{R}$.
We focus on linear arrays with decentralized vehicles, where each vehicle interacts with only a few of its
neighbors. We obtain explicit expressions for necessary conditions for stability
in the cases that a system consists of a periodic arrangement of two or three different
types of vehicles, i.e. configurations as follows: ...2-1-2-1 or ...3-2-1-3-2-1.
Previous literature indicated that the (necessary) condition for stability in the case
of a single vehicle type (...1-1-1) held that the first moment of certain coefficients of
the interactions between vehicles has to be zero. Here, we show that that does not generalize.
Instead, the (necessary) condition in the cases considered is that the first moment \emph{plus
a nonlinear correction term} must be zero.
\end{abstract}

\section{Introduction}

Linear arrays of agents, or particles have been studied in many areas such as flock formations, see \cite{OKUBO19861, Young}
and vehicular platooning, see \cite{Bamieh, Defoort, Lin, swaroop1996string}. We direct our attention
to autonomous vehicular formation in $\mathbb{R}$, namely $n$ vehicles driving on a one-lane road. By autonomous vehicles, we mean that each vehicle does not have any human assistance other than its own set of initial values and a pre-specified set of interaction parameters between its neighbors.

The systems we study are set up as follows. The symbol $\bf z$ is used for the $n$ positions of vehicles
on the line. The equations of motion can be compactly written as
\begin{align}
\dfrac{d}{dt}\begin{pmatrix} \mathbf{z} \\ \dot{\mathbf{z}} \end{pmatrix} =
\begin{pmatrix} 0 & I \\ L_x & L_v \end{pmatrix}
\begin{pmatrix} \mathbf{z} \\ \dot{\mathbf{z}} \end{pmatrix}=
{\bf M} \begin{pmatrix} \mathbf{z} \\ \dot{\mathbf{z}} \end{pmatrix}\;,
\label{eq:laplacian-system}
\end{align}
where $I$ is the $n\times n$ identity, $L_x$ and $L_v$ are $n\times n$ so-called Laplacian
matrices. If all agents are identical, $L_x$ has the following form:
\begin{align}
L_x=g_x\begin{pmatrix} 1 & \rho_{x,1} & \rho_{x,2} & \cdots \\
\rho_{x,-1}& 1 & \rho_{x,1} & \cdots \\
\rho_{x,-2}& \rho_{x,-1} & 1 \cdots \\
\vdots & \vdots && \end{pmatrix},
\label{eq:xlaplacian}
\end{align}
and similar for $L_v$.
Equation (\ref{eq:laplacian-system}) is meant to express the idea that the acceleration of the $k$th
vehicle depends
on the \emph{positions relative to it} of some of his neighbors --- this is expressed through
the matrix $L_x$ --- and on the \emph{velocities relative to it} --- expressed through
$L_v$. Vehicles whose response depends only on positions and velocities \emph{relative to them}
are called \emph{decentralized}. The fact they are decentralized implies
that $L_x$ and $L_v$ have row-sum zero. Hence they share many characteristics with
the usual Laplacian operator (for details, see \cite{LAFF} and \cite{veerman2005flocks}).
Ultimately, what we want to know is the behavior of the flock when the following happens.
For $t\leq 0$ the formation is in equilibrium, that is: $z_i=0$ and $\dot z_i$ is constant.
At $t\geq 0$, the first vehicle changes its velocity, and the others ``try" to follow.

In this paper we continue a line of research started in \cite{cantos2016signal,CantosTrans} and continued in \cite{herman2016transients,herbrych2015dynamics}. This line is distinct from most other work in two
aspects. First, we allow the interactions to be determined by \emph{two non-commuting} Laplacians
as evidenced in equation (\ref{eq:laplacian-system}). This renders the system (generally) non-diagonalizable
and methods of analysis in the previous literature do not apply. Indeed, to be successful one needs
some sort of generalization of the well-known method using \emph{periodic boundary conditions} best known
from applications in physics \cite{ashcroft1976solid}. The idea is that periodic boundary conditions
turn the Laplacian matrices into circulant matrices which can be simultaneously diagonalized
\cite{kra2012circulant}. This renders the system on the circle, at least in principle, soluble by
analytical means. The generalization of the method of periodic boundary conditions to asymmetric
systems is somewhat delicate and
was conjectured and discussed in \cite{cantos2016signal,CantosTrans}, and numerically supported in that
and other work \cite{herman2016transients,herbrych2015dynamics}.

The second aspect in which our work differs from other work concerns the effect of asymmetry in the
Laplacians on the dynamics of large flocks.
This asymmetry leads to non-orthogonal eigenvectors, and this, even for stable systems, can
(and does in many cases) lead to exponentially growing (in the number of agents) transients (see
Definition \ref{def_flock-stability} below).
This effect \emph{cannot be deduced} from the eigenvalues or eigenvectors of $\bf M$ in equation
\ref{eq:laplacian-system}. In the context of traffic, the possibility of such behavior was first pointed
out in \cite{VST2009, VT2010}. Here we illustrate this effect in Figure \ref{fig:picN2}.
From the point of view of traffic, systems with this property are \emph{just as undesirable} as
systems that are unstable in the usual sense (Definition \ref{def_stability} below).
Thus we need to study what systems are stable in \emph{both} senses.

From the above references, one can conclude that the general theory for flocks in the line with two distinct
and asymmetric Laplacians in the line is now reasonably well-established as long as all agents are identical.
In this paper we take the next step and study this theory in the more realistic case where agents
are \emph{not} identical. The ultimate goal here is to give exact conditions for stability for such flocks.
This is still analytically too hard to solve. In order to get some insight in this problem,
we study the dynamics of flocks with periodic arrangements of distinct agents.
In Section \ref{chap:nearest} we study
periodic arrangements of 3 types of agents ($\cdots 3-2-1-3-2-1$) with nearest neighbor interactions and in
Section \ref{chap:nextnearest} the subject is periodic arrangements of 2 types of agents ($\cdots 2-1-2-1$)
with \emph{next} nearest neighbor interactions. For these types of flocks, we develop \emph{necessary}
conditions for stability.

We follow the strategy implied by the aforementioned conjectures. These say that if for large enough $n$ the
system with periodic boundary conditions is unstable, then the system on the line (with non-trivial
boundary conditions) is unstable in either the sense of Definition \ref{def_stability}) or in the
sense of Definition \ref{def_flock-stability}).
Our earlier work for identical agents resulted in the statement that if $\sum_{j\neq 0} \rho_{x,j}j\neq 0$
(see equation (\ref{eq:xlaplacian})), then instability arises. In looking to prove a generalization
of this statement, we, very unexpectedly,
found that for more complicated systems --- presented in this work --- that statement is generally
false. Corollaries \ref{cor:triatomic} and \ref{cor:diatomic} show that in the cases at hand, a
nonlinear correction needs to be taken into account. We note that these formulas show that, surprisingly, \emph{stability is a co-dimension one phenomenon}! Thus, without the help of these formulae, it would be
very difficult to find stable flocks with non-symmetric interactions.

What we just described is a \emph{sufficient} condition for instability or, equivalently, a \emph{necessary}
condition for stability. It is clear that it is not sufficient for stability.
For example, if we give the last agent an infinite mass (setting $g_x=g_v=0$ for this agent), it cannot
change its velocity. Clearly, if the leader changes its velocity, a system with that boundary condition
cannot evolve towards equilibrium. To find a conditions for stability that are necessary \emph{and}
sufficient seems, for now, out of reach. For a more detailed discussion on the influence of
boundary conditions on the global dynamics of a linear system, see \cite{veerman2018spectra}.

In earlier studies of the stability of flocks on the line, one usually made several of the following
assumptions: the number of agents is infinite \cite{swaroop1996string, cook2005conditions}, the interactions
are symmetric or are forward-looking only \cite{ploeg2014controller}, interactions are small
\cite{bamieh2012coherence}, or $L_x$ and $L_v$ are identical, see
\cite{cook2005conditions, Lin, swaroop1996string, renatkins}. Others \cite{bamieh2012coherence, Lin} have
proposed the idea of coherence vehicular formation by local and global feedback and the analysis of
\emph{consensus} dynamics which are systems of first order ordinary differential equations, see also
\cite{hegselmann2002opinion, ren2005survey}. In our paper none of those assumptions are necessary.

For future reference, we define two notions of stability. In consequence of the fact
that $L_v$ and $L_x$ are Laplacians, we see that for arbitrary constant $x_0$ and $v_0$
\eqref{eq:laplacian-system} has an \emph{in formation} solution $z_i=x_0+v_0t$. This is desirable for a flock.
It \emph{does} mean, however, that the matrix associated with this linear system must have
a Jordan block of dimension 2 associated to the eigenvalue 0. In this paper, we will call a
system (linearly) stable if all \emph{other} eigenvalues have strictly negative real part.

\begin{definition} \label{def_stability}
The system \eqref{eq:laplacian-system} is linearly stable,
shortened to stable, if it has one eigenvalue zero with
geometric multiplicity one and algebraic multiplicity two, and all
other eigenvalues have real part less than zero. The system is (linearly)
unstable if at least one eigenvalue has positive real part.
\end{definition}

If in a stable system that is at rest, the leader acquires an initial velocity $v=1$ at $t=0$,
the system will undergo a temporary oscillatory change which will decrease in time.
So eventually the dynamics of a stable system will converge to an \emph{in formation} solution.
The initial oscillatory movement is called a \emph{transient}.
The size of these transients can be measured by the largest change in distance to the leader of any car
at any time \cite{VT2010,TVS2012}. See \cite{CantosTrans} for a more detailed discussion of the definition of \emph{flock stability}.

\begin{definition}\label{def_flock-stability}
The system \eqref{eq:laplacian-system} is flock stable if it is linearly stable \emph{and}
if transients grow less than exponentially fast in the number $n$ of vehicles. It is called
 flock unstable if the growth is exponential as $n$ tends to infinity.
\end{definition} 

\color{black}

\section{Periodic Arrangements with Nearest Neighbor Interactions.}
\label{chap:nearest}

\noindent
Linear flocks in $\mathbb{R}$ of \emph{identical} agents \color{black} have been thoroughly studied
(\cite{CantosTrans, herman2016transients, herbrych2015dynamics}). The necessary condition for stability is
that the first moment (or $\sum_{i =1}^3 \rho_{x,i}i$) of the coefficients of the spatial Laplacian must
be zero. For flocks of type ...2-1-2-1, the same is true. Details of the latter can be found in
\cite{baldivieso2019}. Here we will look at the arrangement ...3-2-1-3-2-1. Thus we consider of linear
arrays with $N=3n$ ($n$ of each type) vehicles in which each vehicle interacts with its nearest neighbors.
The quantities $z^{(i)}_j$ are the deviations from their equilibrium positions. The quantities $\dot{z}^{(i)}_j$, $i = 1, 2, 3$,
and $j= 1 \dots n$ are their derivatives with respect to time.

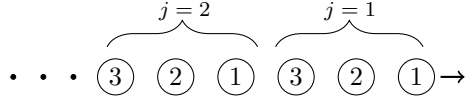
\begin{figure}
\centering
\begin{tikzpicture}[scale = .8]
\draw (0,0) circle [radius = .3];
\draw (1,0) circle [radius = .3];
\draw (2,0) circle [radius = .3];

\draw (-1,0) circle [radius = .3];
\draw (-2,0) circle [radius = .3];
\draw (-3,0) circle [radius = .3];

\draw[fill = black] (-3.7,0) circle [radius = .04];
\draw[fill = black] (-4.2,0) circle [radius = .04];
\draw[fill = black] (-4.7,0) circle [radius = .04];

\draw [thick,->] (2.4,0) -- (2.8,0);

\draw (2,0) node {$1$};
\draw (1,0) node {$2$};
\draw (0,0) node {$3$};
\draw (-1,0) node {$1$};
\draw (-2,0) node {$2$};
\draw (-3,0) node {$3$};

\draw [decorate,decoration={brace,amplitude=10pt},xshift=-4pt,yshift=0pt] (-0.2,.5) -- (2.25,.5);
\draw  (.9,1.1) node {\footnotesize $j = 1$};

\draw [decorate,decoration={brace,amplitude=10pt},xshift=-4pt,yshift=0pt] (-3,.5) -- (-.5,.5);
\draw  (-1.85,1.1) node {\footnotesize $j = 2$};
\end{tikzpicture}
\caption{\emph{Periodic arrangement of flocks with three types of vehicles, labeled by 1,2, and 3.
At time $t=0$, the first vehicle start moving to the right.}
\label{fig-arrangement}}
\end{figure}

The equations of motions for each type of particle are (see Figure \ref{fig-arrangement}):
\begin{align} \label{three_cars_model}
\begin{split}
\ddot{z}\one_j & = g\one_x  \left(z\one_j + \rho\one_{x,1} z\two_j+ \rho\one_{x,-1}z\three_{j-1} \right) + g\one_v \left(\dot z\one_j+ \rho\one_{v,1} \dot z\two_j + \rho\one_{v,-1} \dot z\three_{j-1}\right) \\
\ddot{z}\two_j & = g\two_x \left(z\two_j + \rho\two_{x,1} z\three_{j} + \rho\two_{x,-1} z\one_j \right) + g\two_v \left(\dot z\two_j + \rho\two_{v,1}\dot z\three_{j} + \rho\two_{v,-1}\dot z\one_j \right) \\
\ddot{z}\three_j & = g\three_x \left(z\three_j + \rho\three_{x,1} z\one_{j+1} + \rho\three_{x,-1} z\two_j \right) + g\three_v \left(\dot z\three_j + \rho\three_{v,1}\dot z\one_{j+1} + \rho\three_{v,-1}\dot z\two_j \right)
\end{split} \quad .
\end{align}
\noindent
We assume the flocks to be \emph{decentralized}, that is: the acceleration of an individual depends
only on observation \emph{relative} to that individual. For example, the first of the equations
in equation \eqref{three_cars_model}, should be thought of as:
\begin{align*}
\ddot{z}\one_j = g\one_x  \left[\rho\one_{x,1} \left(z\two_j-z\one_j\right)+
\rho\one_{x,-1}\left(z\three_{j-1}-z\one_j\right)\right] +
g\one_x  \left[\rho\one_{x,1} \left(\dot z\two_j-\dot z\one_j\right)+
\rho\one_{x,-1}\left(\dot z\three_{j-1}-\dot z\one_j\right)\right] \;.
\end{align*}
This leads to the following constraints: for $i\in\{1,2,3\}$
\begin{align} \label{constraints3}
\rho^{(i)}_{x,1} + \rho^{(i)}_{x,-1} = -1, ~~~~~~ \rho^{(i)}_{v,1} + \rho^{(i)}_{v,-1} = -1\quad .
\end{align}
\noindent
We will assume that $g\one_x, g\two_x, g\three_x, g\one_v, g\two_v$, and $g\three_v$ are real numbers.

According to the strategy described in the introduction, instability in the system with periodic
boundary condition will imply some form of instability (Definition \ref{def_stability} or Definition
\ref{def_flock-stability}) in the system on the real line if $N$ is large. Thus our task reduces
to deriving a criterion for instability for the system, given periodic boundary
conditions. The system subject to periodic boundary conditions is described as follows.
\begin{align}
\dfrac{d}{dt}\begin{pmatrix}
\mathbf{z}\one \\
 \mathbf{z}\two \\
 \mathbf{z}\three\\
 \dot{\mathbf{z}}\one \\
 \dot{\mathbf{z}}\two\\
  \dot{\mathbf{z}}\three
 \end{pmatrix} =
 \left( \begin{array}{c|c|c|c|c|c}\mathbf{0} & \mathbf{0} & \mathbf{0} & \mathbf{I} & \mathbf{0} & \mathbf{0} \\ \hline
\mathbf{0} & \mathbf{0} & \mathbf{0} & \mathbf{0} & \mathbf{I} & \mathbf{0} \\ \hline
\mathbf{0} & \mathbf{0} & \mathbf{0} & \mathbf{0} & \mathbf{0} & \mathbf{I} \\ \hline
g_x\one\mathbf{I} & g_x\one \rho\one _{x,1} \mathbf{I}& g_x\one \rho\one_{x,-1}\mathbf{P}_- & g_v\one\mathbf{I} & g_v\one \rho\one _{v,1}\mathbf{I} & g_v\one \rho\one_{v,-1}\mathbf{P}_- \\ \hline
g_x\two\rho\two _{x,-1} \mathbf{I} & g_x\two\mathbf{I} & g_x\two \rho\two _{x,1}\mathbf{I} & g_v\two \rho\two _{v,-1}\mathbf{I} & g_v\two\mathbf{I} & g_v\two \rho\two _{v,1}\mathbf{I}\\ \hline
g_x\three\rho\three_{x,1}\mathbf{P}_+ & g_x\three \rho\three _{x,-1}\mathbf{I} & g_x\three\mathbf{I}  & g_v\three\rho\three_{v,1}\mathbf{P}_+ & g\three _v \rho\three _{v,-1}\mathbf{I} & g_v\three\mathbf{I}
\end{array}\right)
\begin{pmatrix}
\mathbf{z}\one \\
\mathbf{z}\two \\
\mathbf{z}\three \\
\dot{\mathbf{z}}\one \\
\dot{\mathbf{z}}\two\\
\dot{\mathbf{z}}\three
\end{pmatrix} \;,
\label{eq:linear_sys3}
\end{align}
\noindent
where $\mathbf{P}_+$ and its inverse $\mathbf{P}_-$ are \color{blue}$n\times n$ \color{black} permutations matrices
\begin{align}
\mathbf{P}_+ =
\begin{pmatrix}
0 		&	1		&	0	&	\cdots		&	0\\
0		&	0		&	1		&	\ddots		&	\vdots \\
\vdots	&	\ddots	&	\ddots		&	\ddots		&	0 \\
0		&	\ddots	&		\ddots		&		0		&  1\\
1		&	0		&		\cdots		&	0		&	0	
\end{pmatrix}~,~~~
\mathbf{P}_- =
\begin{pmatrix}
0 		&	0		&	\cdots	&	0		&	1\\
1		&	0		&	0		&	\ddots		&	0\\
0		&	1		&	\ddots		&	\ddots		&	\vdots \\
\vdots	&	\ddots	&		\ddots		&		0		&  0\\
0		&	\cdots	&		0		&	1			&	0	
\end{pmatrix}\;.
\label{permut_matrices}
\end{align}
We will abbreviate equation \eqref{eq:linear_sys3} simply as
\begin{align}
\dfrac{d}{dt}\begin{pmatrix} \mathbf{z} \\ \dot{\mathbf{z}} \end{pmatrix} = \mathbf{M}
\begin{pmatrix} \mathbf{z} \\ \dot{\mathbf{z}} \end{pmatrix}\;.
\label{eq:simple-eqn}
\end{align}

\begin{definition} From now on, we set $\phi_m=\frac{2\pi m}{n}$, $m\in\{0,\cdots n-1\}$.
When there is no ambiguity, we will often drop the subscript from $\phi_m$. We let ${\bf v}_m$ be
the $n$-vector whose $j$th component equals $e^{ij\phi_m}$.
\label{def:evecs}
\end{definition}

\begin{prop} The eigenvalues $\nu$ and associated eigenvectors ${\bf u}_\nu(\phi_m)$ of $\mathbf{M}$ satisfy
\begin{align*}
{\bf u}_\nu(\phi_m)=\begin{pmatrix}\e_1\mathbf{v}_m, \e_2\mathbf{v}_m, \e_3\mathbf{v}_m, \nu \e_1\mathbf{v}_m, \nu \e_2\mathbf{v}_m, \nu \e_3\mathbf{v}_m \end{pmatrix}^T\;.
\end{align*}
For each $m\in\{0,\cdots n-1\}$ given, there are six eigenpairs (counting multiplicity)
determined by solving the following equation for $\nu$ and $\epsilon_i$:
\begin{align*}
\begin{pmatrix}
g_x\one + \nu g_v\one - \nu^2  & g\one _x\rho\one_{x,1} +  \nu g_v\one\rho\one_{v,1} &
\left(g\one _x\rho\one_{x,-1} +  \nu g_v\one \rho\one_{v,-1}\right)e^{-i\phi} \\
g_x\two\rho\two_{x,-1} + \nu g_v\two\rho\two_{v,-1} & g_x\two + \nu g_v\two - \nu^2  & g_x\two\rho\two_{x,1}  + \nu g_v\two\rho\two_{v,1} \\
\left(g_x\three\rho\three_{x,1} + \nu g_v\three\rho\three_{v,1}\right)e^{i\phi} & g_x\three\rho\three_{x,-1} +  \nu g_v\three\rho\three_{v,-1}  & g\three_x + \nu g_v\three - \nu^2
\end{pmatrix}
\begin{pmatrix}
\e_1\\
\e_2\\
\e_3
\end{pmatrix}
=
\begin{pmatrix}
0\\ 0\\ 0
\end{pmatrix}\;.
\end{align*}
\label{prop:charpoly}
\end{prop}

\begin{proof}
From equations \eqref{eq:linear_sys3} and \eqref{eq:simple-eqn}, we see that an eigenvector $\begin{pmatrix}\ {\bf z}\\ \dot {\bf z}\end{pmatrix}$
associated to the eigenvalue $\nu$ satisfies
${\bf\dot z}= \nu {\bf z}$. Now $\mathbf{P}_+^n=I$, and so $e^{i\phi_m}$ and ${\bf v}_m$ are the eigenvalues and
eigenvectors of $\mathbf{P}_+$, and $e^{-i\phi_m}$ and $v_m$ of $\mathbf{P}_-$. Then by substituting ${\bf u}_\nu$ into
\eqref{eq:linear_sys3}, one sees that these are the eigenvectors of ${\bf M}$.

For the second part, note that the eigenvector $u$ derived above has 4 unknowns. \color{black}
We can write
\begin{align} \label{eigvec31}
\mathbf{M}\mathbf{u} = \nu \mathbf{u}\;,
\end{align}
substitute $u$ of the first part, and substitute that in equation
\eqref{eq:linear_sys3}. We obtain three non-trivial equations (from the last three lines of
\eqref{eq:linear_sys3}), which can be simplified and rearranged to give the second part of
the proposition. (By linearity, if $u$ is a solution, then so is any multiple of $u$.
Thus 3 equations is enough.)
\color{black}
\end{proof}

In short, we can find all eigenpairs by setting to zero the determinant of the matrix in Proposition
\ref{prop:charpoly}. We obtain a polynomial $Q$ of degree six in $\nu$. In its full glory, the polynomial
is more than a little cumbersome. From now on, we take superscripts $g$ and $\rho$ modulo 3.
For example, $g_x^{(5)}=g\two_x$. This allows us to manage the expressions a little better.

\begin{definition} Let $a$, $b$, $c$, $d$, and $t$ be real numbers, define
\begin{align*}
D(a,b,c;t) \equiv abc(e^{it}-1)-(1+a)(1+b)(1+c)(e^{-it}-1), ~~\text{and}~~E(a,b,c,d) \equiv ab(1+c+cd)\;.
\end{align*}
\label{def:C(abc)}
\end{definition}
The following Lemma is the result of substantial bookkeeping which we leave to the reader.
\begin{lem} When $\phi=0$, the matrix of Proposition \ref{prop:charpoly} has determinant
$Q(\nu,\phi=0)$ equal to
\begin{equation*}
\begin{array}{rl}
-\nu^2 & \sum_{i=1}^3\, E(g_x^{(i)},g_x^{(i+1)},\rho_{x,1}^{(i)},\rho_{x,1}^{(i+1)})\\
-\nu^3 & \sum_{i=1}^3\,\left[E(g_x^{(i)},g_v^{(i+1)},\rho_{x,1}^{(i)},\rho_{v,1}^{(i+1)})+
E(g_v^{(i)},g_x^{(i+1)},\rho_{v,1}^{(i)},\rho_{x,1}^{(i+1)})\right]\\
-\nu^4 & \sum_{i=1}^3\,\left[g_x^{(i)}+E(g_x^{(i)},g_x^{(i+1)},\rho_{x,1}^{(i)},\rho_{x,1}^{(i+1)})\right]+\nu^5 \sum_{i=1}^3\, g_x^{(i)} -\nu^6
\end{array} \quad .
\end{equation*}
The full expression of the constant term of $Q(\nu,\phi)$ is $a_0(\phi)$, where $a_0(\phi)=g\one_x g\two_x g\three_x D(\rho\one_{x,1},\rho\two_{x,1},\rho\three_{x,1};\phi)$.
\label{lem:bookkeeping}
\end{lem}

To simplify the statement of the main results further, we also need the following definition.

\begin{definition} For $j$ and $k$ positive, we define $~\alpha^{(k)}_{x,j} \equiv \rho^{(k)} _{x,j} + \rho^{(k)}_{x,-j} ~~\text{and}~~ \beta^{(k)}_{x,j} \equiv \rho^{(k)} _{x,j} - \rho^{(k)}_{x,-j}$.

Because of the constraint \eqref{constraints3}, the $\alpha$'s are equal to -1 in this case (but
not in the next section).
\label{def:alphas-betas}
\end{definition}

\begin{thm} If any of the following conditions are satisfied, then for large $N$, the system
given by \eqref{three_cars_model} on the circle is not (linearly) stable.

$(i)$ $g\one_x = 0$ or $g\two_x = 0$, or $g\three_x = 0$.

$(ii)$ $\sum_{i=1}^3\, E(g_x^{(i)},g_x^{(i+1)},\rho_{x,1}^{(i)},\rho_{x,1}^{(i+1)})
= 0$.

$(iii)$ $ \displaystyle g\one_x g\two_x g\three_x\,\left[\sum_{i =1}^3 \beta^{(i)}_{x,1} +
\prod_{i = 1}^3 \beta^{(i)}_{x,1}\right] \ne 0$.
\label{theorem_triatomic}
\end{thm}
\color{black}

\begin{proof} We start with part $(i)$. Suppose for example that $g\one_x = 0$. Then the first row of the matrix
in Proposition \ref{prop:charpoly} has a factor $\nu$. Since the determinant is a linear function
of the rows, it follows that the determinant of that matrix also has a factor $\nu$.
This implies that the zero eigenvalue has multiplicity of at least $N$, contradicting
Definition \ref{def_stability}. Suppose $(ii)$ holds. Then from Lemma \ref{lem:bookkeeping} we see that
for $\phi=0$, we get multiplicity 3 for the eigenvalue zero.
This violates Definition \ref{def_stability}. Part $(iii)$ The eigenvalues of $\bf M$ in equation (\ref{eq:simple-eqn}) are the roots of $Q(\nu,\phi)=\sum_{i=2}^6\,a_i(\phi)z^i +2a_1(\phi)z+a_0(\phi),$
where the $a_i(0)$ are given in the first part of Lemma \ref{lem:bookkeeping}. The second part
of that lemma states that $a_0'(0)=g\one_x g\two_x g\three_x \, \dfrac{\partial}{\partial \phi}D(\rho\one_{x,1},\rho\two_{x,1},\rho\three_{x,1};\phi)\big|_{\phi=0}\;.$
According to Proposition \ref{prop:appendix}, the system is unstable if $~a_0(0)=a_1(0)=0 \quad \textrm{and} \quad a_2(0)\neq 0 \quad \textrm{and} \quad a_0'(0)\neq 0.$
Substituting $\rho^{(i)}_{x,1} = \dfrac{1}{2}\left(\beta^{(i)}_{x,1} - 1\right)$ in $a_0'(0)$
(using Definition \ref{def:alphas-betas} and the constraints \eqref{constraints3}) yields part
$(iii)$.
\color{black}
\end{proof}

According to this proof, in case (i) and (ii), the system has too many eigenvalues with zero real part
(this is called marginal stability), but not necessarily with positive real part. \emph{Instability}
in the sense of Definition \ref{def_stability} requires at least one eigenvalue with
\emph{positive} real part. The conjectures in \cite{CantosTrans} state that instability
on the circle implies some form of \emph{instability} on the line.

\begin{cor} Using the conjectures of \cite{CantosTrans}, we obtain that if
$\sum_{i=1}^3\, E(g_x^{(i)},g_x^{(i+1)},\rho_{x,1}^{(i)},\rho_{x,1}^{(i+1)})\ne 0$ and
\begin{equation*}
 \displaystyle g\one_x g\two_x g\three_x\,\left[\sum_{i =1}^3 \beta^{(i)}_{x,1} + \prod_{i = 1}^3 \beta^{(i)}_{x,1}\right] \neq  0\;,
\end{equation*}
\noindent
then the system on the line given by (\ref{three_cars_model})
has some form of instability (Definitions \ref{def_stability} or \ref{def_flock-stability}).
\label{cor:triatomic}
\end{cor}

\noindent
A generalization of Proposition \ref{prop:appendix} shows that the condition $a_2(0)\neq 0$
is not necessary to guarantee the presence of eigenvalues with positive real part for large $n$.
Since $a_2(0)\neq 0$ corresponds to the first condition in the corollary (see Lemma \ref{lem:bookkeeping}), we drop that condition from now on. Details will appear elsewhere \cite{LyonsVeerman}.
\begin{figure}
\begin{center}
\begin{subfigure}{.5\textwidth}
\centering
\includegraphics[width=8.8cm,height=8.5cm]{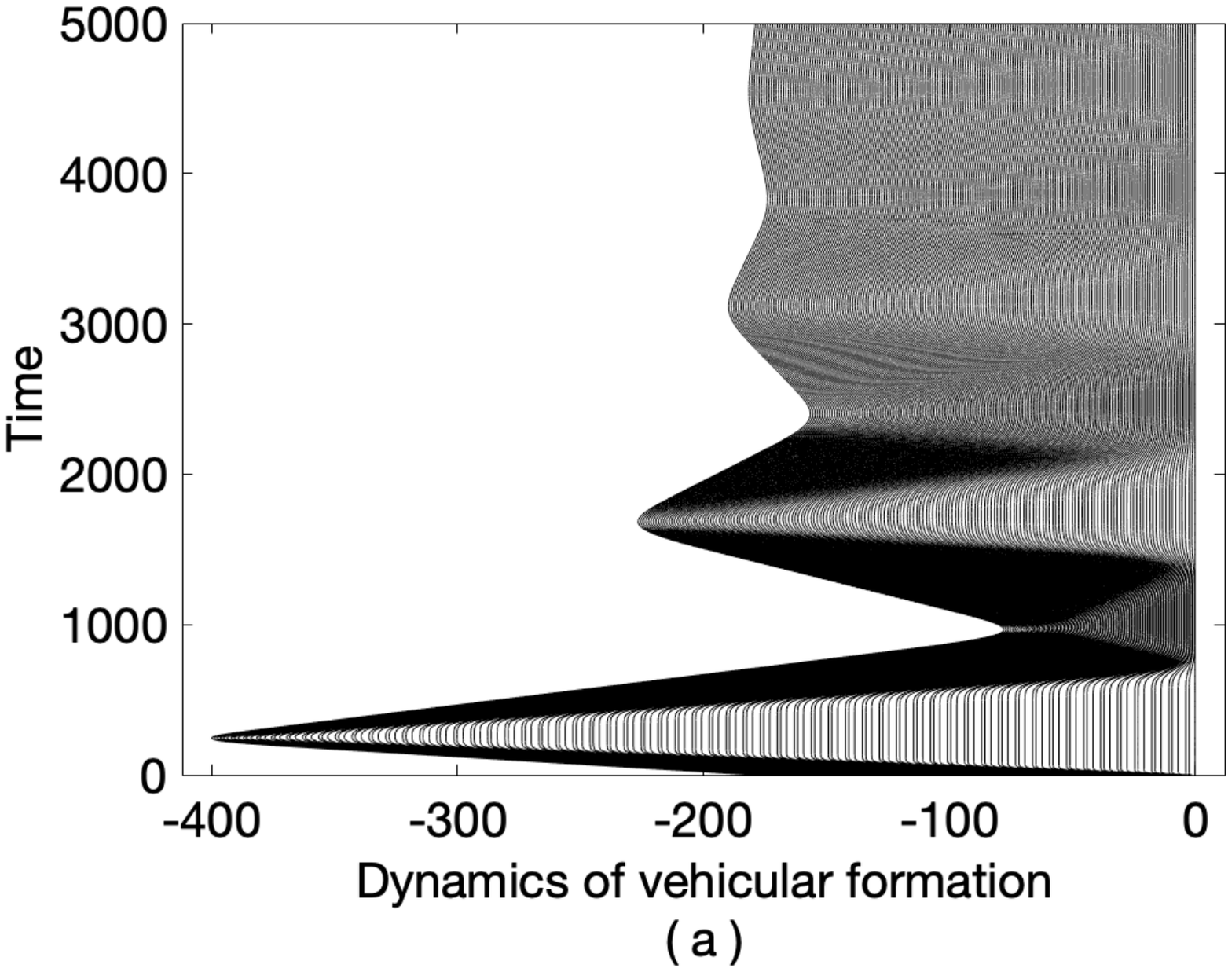}
\end{subfigure}%
\begin{subfigure}{.5\textwidth}
\centering
\includegraphics[width=8.8cm,height=8.5cm]{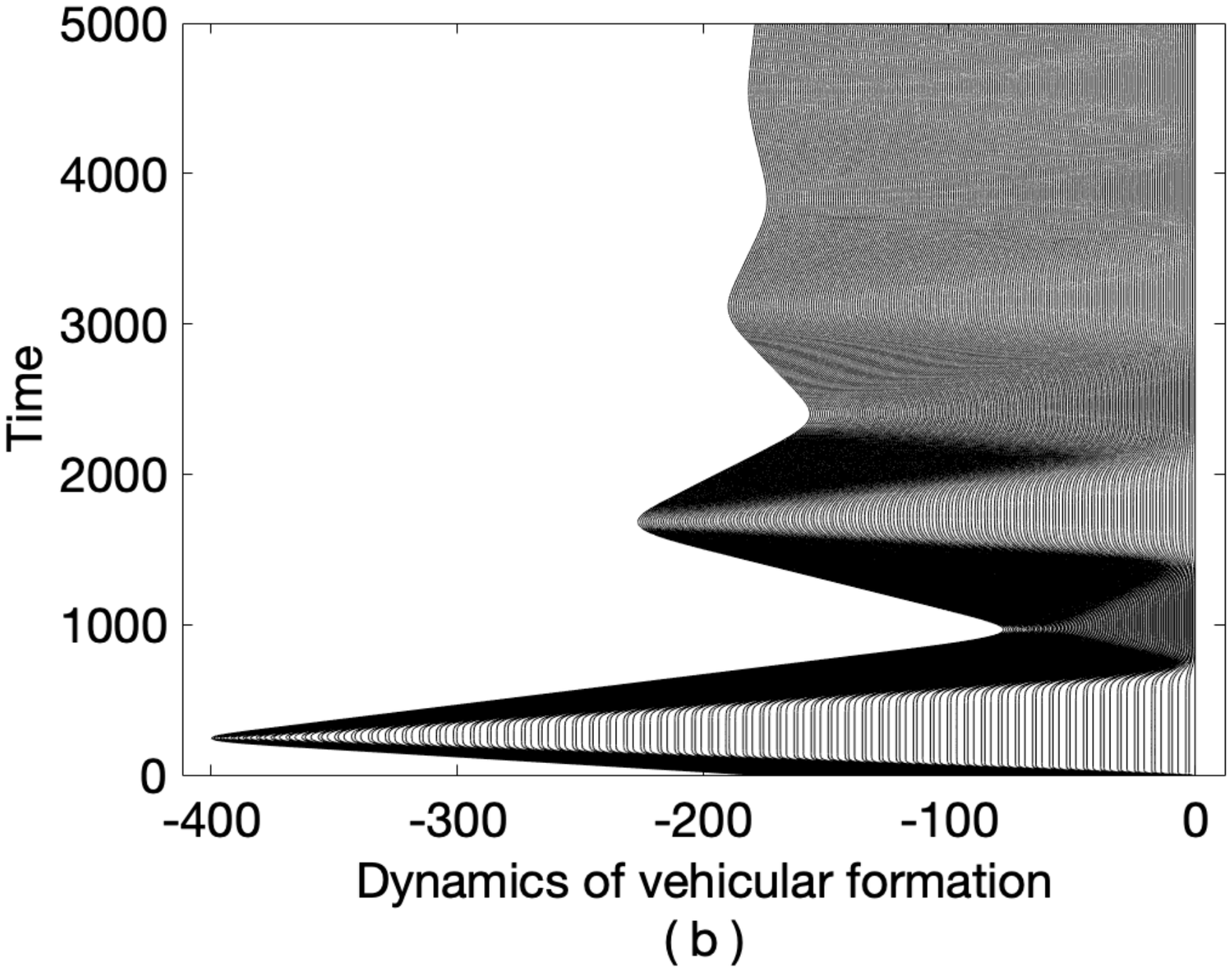}
\end{subfigure}
\vspace{-4\baselineskip}
\caption{\emph{Dynamics of a stable system with $N=180$ vehicles with  $\boldsymbol{\rho} =-(0.6,0.8,0.142857,0.3,0.3,0.3)$ and $\boldsymbol{g} = -(1,1,1,1.3,1.3,1.3)$.} (a) \emph{Boundary Condition Type I. Maximum amplitude of $-221.0$ at $t=244.6$.}
(b) \emph{Boundary Condition Type II. Maximum amplitude of $-220.8$ at $t=-244.4$. }}
\label{fig:picN1}
\end{center}
\end{figure}

\begin{figure}
\begin{center}
\begin{subfigure}{.5\textwidth}
\centering
\includegraphics[width=8.8cm,height=8.5cm]{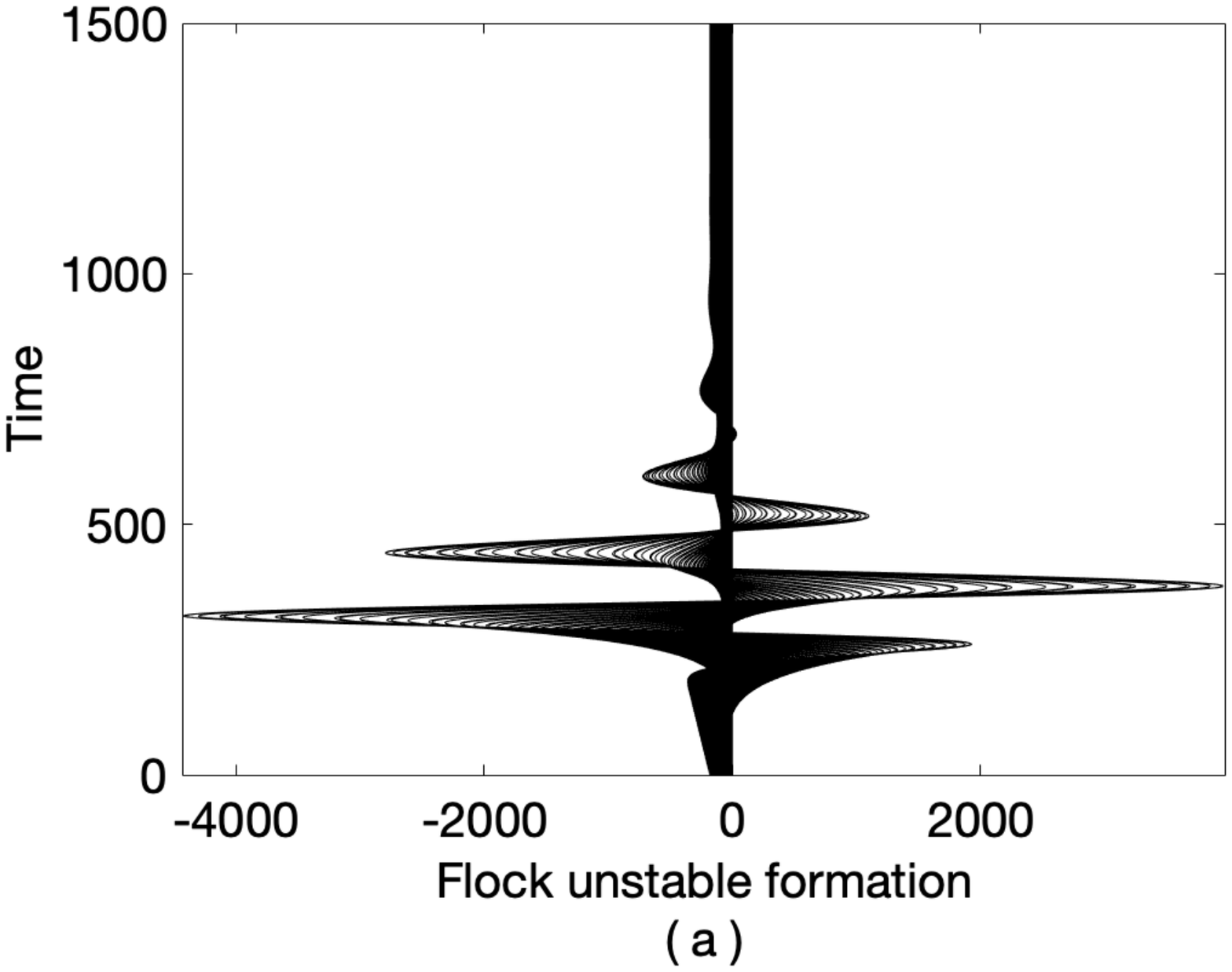}
\end{subfigure}%
\begin{subfigure}{.5\textwidth}
\centering
\includegraphics[width=8.7cm,height=8.5cm]{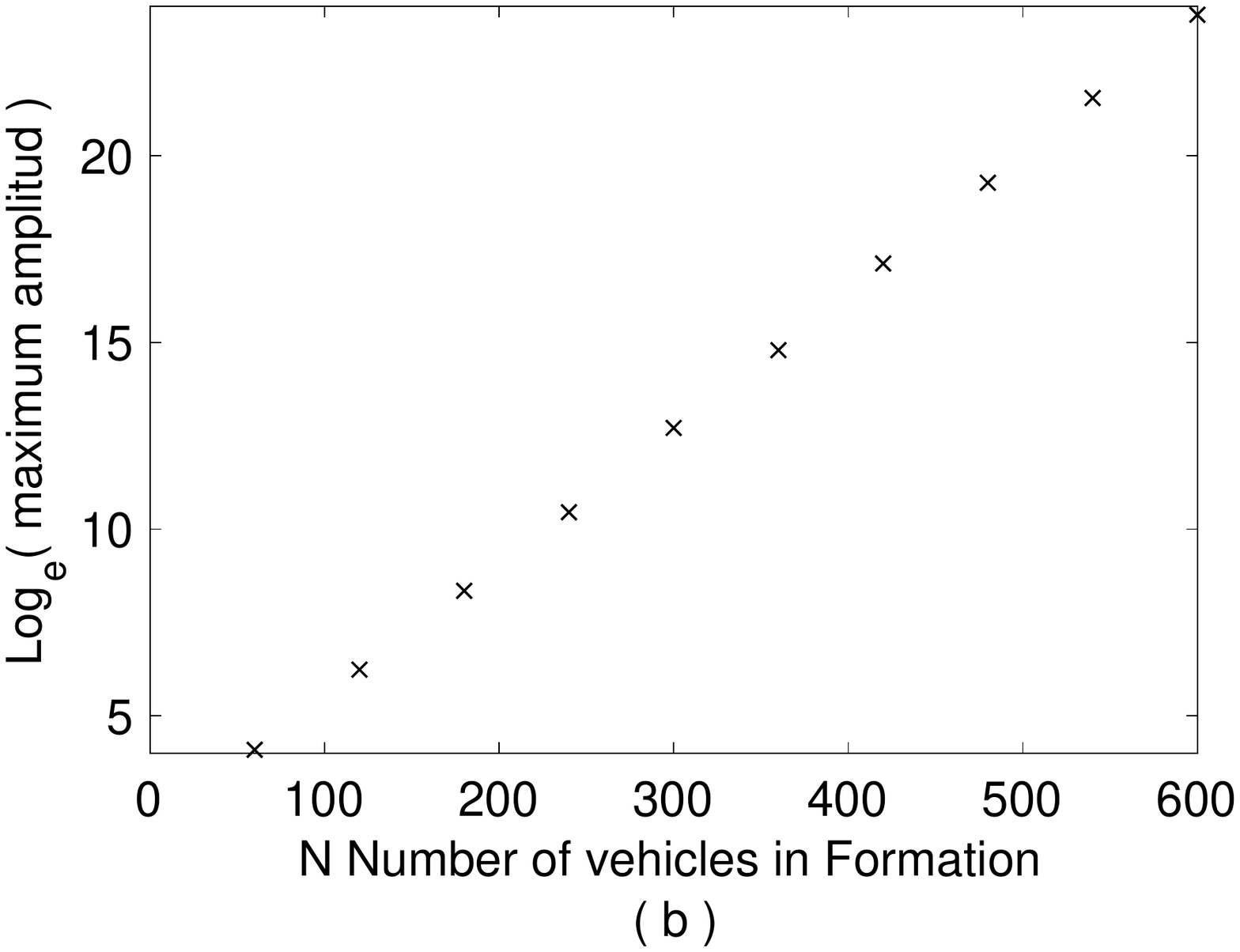}
\end{subfigure}
\vspace{-4\baselineskip}
\caption{\emph{Dynamics of flock unstable systems.} (a) \emph{Dynamics of a
flock unstable system of $180$ vehicles} with  $\boldsymbol{\rho} =-(0.6,0.8,0.10,0.3,0.3,0.3)$ and $g_{x,v} = -(1,1,1,1.3,1.3,1.3)$. (b) Same coefficients as (a), we see that \emph{amplitude of flock unstable systems increase exponentially as the number $N$ of cars increases.}}
\label{fig:picN2}
\end{center}
\end{figure}

We perform simulations
to see if this conclusion is borne out by simulations on the real line (independent
of reasonable boundary conditions).
Similar to what was done in \cite{herbrych2015dynamics}, we consider two sets of boundary conditions.
We will call them Type I and Type II boundary conditions.
Since we want to maintain the centralized character of the systems, both sets of boundary conditions
must maintain the ``Laplacian" property, namely that row-sums of each Laplacian are zero.
Type I adjusts the central coefficients $\rho^{(i)}_{x,0} $, and $\rho^{(i)}_{v,0}$ on the boundaries
as: $\ddot{z}\one_{1} = 0 ~~ \text{and}~~\ddot{z}\three_n = g\three_x \left(-\rho\three_{x,-1}z\three_n + \rho\three_{x,-1} z\two_n \right) + g\three_v \left(- \rho\three_{v,-1}\dot z\three_n + \rho\three_{v,-1}\dot z\two_n \right)$.

In Type II boundary conditions, we keep the central coefficients $\rho^{(i)}_{x,0} $, and
$\rho^{(i)}_{v,0}$ equal to 1 and we adjust the remaining coefficients accordingly: $\ddot{z}\one_{1} = 0 ~~\text{and}~~\ddot{z}\three_n = g\three_x \left(z\three_n - z\two_n \right) + g\three_v \left(\dot z\three_n - \dot z\two_n \right)$.

We run simulations of the system in $\mathbb{R}$ considering these two boundary conditions with
initial condition:
\begin{align*}
z^{(i)}_k(0) = \dot z^{(i)}_k(0) = 0 \quad \textrm{\underline{except}} \quad
\dot{z}^{(1)}_1(0) = 1\;.
\end{align*}
To shorten our notation, let us write the parameters of our system as $\boldsymbol{\rho} = (\rho\one_{x,1},\rho\two_{x,1}, \rho\three_{x,1}, \rho\one_{v,1},\rho\two_{v,1}, \rho\three_{v,1})$, and $\boldsymbol{g} = (g\one_x,g\two_x,g\three_x,g\one_v,g\two_v,g\three_v)$. \color{black}
 Figure \ref{fig:picN1}(a) and \ref{fig:picN1}(b) are numerical simulations of $n=60$ vehicles of each type on the line with parameters given in the Figure.
These parameters satisfy Corollary \ref{cor:triatomic}. 
 Thus $\sum_{i =1}^3 \beta^{(i)}_{x,1}=-0.0858$, while
$\sum_{i =1}^3 \beta^{(i)}_{x,1} + \prod_{i = 1}^3 \beta^{(i)}_{x,1}=0$.
So it is far from satisfying the first, but satisfies the stability condition derived
in this section. From the figures, it is apparent that the system is stable in the sense
both definitions \color{black}, and that
the outcome is largely independent of the type of boundary condition.

On the other hand, Figure \ref{fig:picN2}(a) shows the typical dynamics of
a \emph{flock unstable} system. We see that around time $t=300$, one of the leader-agent distances
is roughly 4000 units from its desired value. The largest such distance in an otherwise stable system is called the \emph{magnitude of the transient} \cite{VT2010,TVS2012}. The stability of the system guarantees
its ultimate return to equilibrium for large $t$. What happens here is that as the number $N$
of agents grows, the magnitude of the transient grows exponentially in $N$. This is illustrated in
Figure \ref{fig:picN2}(b) where the logarithm of the magnitude of the transients is plotted
as function of the number of vehicles. Because of this exponential increase, we get very large
transients even for moderate $N$. Obviously, for traffic purposes such systems are undesirable. The parameters of  Figure \ref{fig:picN2}(a) are similar to Figure \ref{fig:picN1}, except $\boldsymbol{\rho} =-(0.6,0.8,0.10,0.3,0.3,0.3)$ satisfying $\sum_{i =1}^3 \beta^{(i)}_{x,1}=0$, but not the condition derived in this section.

\section{Periodic Arrangements with Next Nearest Neighbor Interactions}
\label{chap:nextnearest}

Next nearest neighbor interaction means that a vehicle can see up to two vehicles in front and behind it.
Although such systems with identical vehicles were included in \cite{cantos2016signal}, they were more thoroughly
studied in \cite{herbrych2015dynamics}, where it was shown that for certain parameter values, these systems
can generate so-called reflectionless waves.
In this section, we consider the stability problem for the more complicated case of
flocks of type ...2-1-2-1 with next nearest neighbor interaction (see Figure \ref{fig-arrangement2}).
With $n$ agents of each type, we have $N=2n$ agents all together.

Similarly as the previous chapter, for $k = 1,\dots,n$, the relevant equations of motion become
\begin{align} \label{NN_diatomic_sys}
\begin{split}
\ddot{z}\one_k & = g\one_x({z}\one_k+\rho\one_{x,1}{z}\two_k+\rho\one_{x,-1}{z}\two_{k-1}+
\rho\one_{x,2}{z}\one_{k+1}+\rho\one_{x,-2}{z}\one_{k-1})\\
& ~~~+ g\one_v(\dot{z}\one_k+\rho\one_{v,1}\dot{z}\two_k+\rho\one_{v,-1}\dot{z}\two_{k-1}+
\rho\one_{v,2}\dot{z}\one_{k+1}+\rho\one_{v,-2}\dot{z}\one_{k-1})\\
\ddot{z}\two_k &= g\two_x(z\two_k+\rho\two_{x,-1}z\one_k+\rho\two_{x,1}z\one_{k+1}+
\rho\two_{x,-2}z\two_{k-1}+\rho\two_{x,2}z\two_{k+1}) \\
&~~~+g\two_v(\dot{z}\two_k+\rho\two_{v,-1}\dot{z}\one_k+\rho\two_{v,1}\dot{z}\one_{k+1}+
\rho\two_{v,-2}\dot{z}\two_{k-1}+\rho\two_{v,2}\dot{z}\two_{k+1})
\end{split} \quad .
\end{align}
Because we assume the equations are decentralized, we get the constraints:
\begin{align}
\sum_{j=-2, j\ne 0}^2\rho^{(i)}_{x,j} = -1\; , \quad \sum_{j=-2, j\ne 0}^2\rho^{(i)}_{v,j} = -1\;.
\label{eq:constraint_minus_one}
\end{align}

As in Section \ref{chap:nearest}, we formulate the system with periodic boundary conditions and
investigate its stability. \color{black} That system can be written more compactly as \eqref{eq:simple-eqn}. But now $\bf M$ is given by
\begin{align}
\left( \begin{array}{c|c|c|c}\mathbf{0} & \mathbf{0} & \mathbf{I} & \mathbf{0}\\ \hline
\mathbf{0} & \mathbf{0} & \mathbf{0} & \mathbf{I}\\ \hline
g_x\one\mathbf{B}\one_x & g_x\one\mathbf{A}_x\one & g_v\one\mathbf{B}\one_v & g_v\one\mathbf{A}_v\one \\ \hline
g_x\two\mathbf{A}_x\two & g_x\two\mathbf{B}\two_x & g_v\two\mathbf{A}_v\two & g_v\two\mathbf{B}\two_v
\end{array}\right)
 \;.
\label{linear_sys}
\end{align}
The $n\times n$ matrices $\bf A$ and $\bf B$ are defined below in terms of the permutation matrices
${\bf P}_\pm$ of \eqref{permut_matrices}. All matrices ${\bf A}$
and ${\bf B}$ are circulant $n\times n$ \color{black} matrices. Thus in the basis ${\bf v}_m$ given in Definition
\ref{def:evecs} is an eigenbasis for all, and the eigenvalues are trivial to compute.
We list all matrices and their eigenvalues in \eqref{matrices_NNN}.

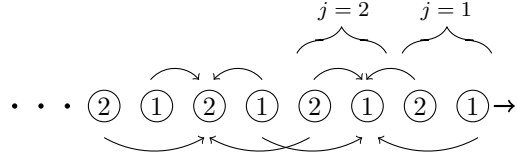
\begin{figure}[h]
\centering
\begin{tikzpicture}[scale = .7]
\draw (0,0) circle [radius = .3];
\draw (1,0) circle [radius = .3];
\draw (2,0) circle [radius = .3];
\draw (3,0) circle [radius = .3];

\draw (-1,0) circle [radius = .3];
\draw (-2,0) circle [radius = .3];
\draw (-3,0) circle [radius = .3];
\draw (-4,0) circle [radius = .3];

\draw[fill = black] (-4.7,0) circle [radius = .04];
\draw[fill = black] (-5.2,0) circle [radius = .04];
\draw[fill = black] (-5.7,0) circle [radius = .04];

\draw [thick,->] (3.4,0) -- (3.8,0);

\draw (3,0) node {$1$};
\draw (2,0) node {$2$};
\draw (1,0) node {$1$};
\draw (0,0) node {$2$};
\draw (-1,0) node {$1$};
\draw (-2,0) node {$2$};
\draw (-3,0) node {$1$};
\draw (-4,0) node {$2$};

\draw[->]  (-1.,-.6) arc (-120:  -60:1.9cm);
\draw[<-]  (1.2,-.6) arc (-120:  -60:1.9cm);

\draw[<-]  (.9,.5) arc (40:  140:.6cm);
\draw[->]  (1.9,.5) arc (40:  140:.6cm);

\draw[->]  (-4.,-.6) arc (-120:  -60:1.9cm);
\draw[<-]  (-2,-.6) arc (-120:  -60:1.9cm);

\draw[<-]  (-2.2,.5) arc (40:  140:.6cm);
\draw[->]  (-1,.5) arc (40:  140:.6cm);

\draw [decorate,decoration={brace,amplitude=10pt},xshift=-4pt,yshift=0pt] (1.8,1) -- (3.5,1);
\draw  (2.5,1.8) node {\footnotesize $j = 1$};

\draw [decorate,decoration={brace,amplitude=10pt},xshift=-4pt,yshift=0pt] (-0.2,1) -- (1.5,1);
\draw  (.55,1.8) node {\footnotesize $j = 2$};

\end{tikzpicture}
\caption{\emph{Periodic arrangement of flock with two types of vehicles, labeled by 1 and 2.
Each vehicle uses information from four others; the arrows indicate information flow.
At time $t=0$, the first vehicle start moving to the right.}
\label{fig-arrangement2}}
\end{figure}


\begin{equation}
\begin{array}{llll}
\mathbf{A}_x \one &= \rho\one_{x,1}\mathbf{I}+\rho\one_{x,-1}\mathbf{P}_- \;;&  \lambda_{x}\one(\phi)
&= \rho\one _{x,1} +\rho\one_{x,-1}e^{-i \phi}\;. \\
\mathbf{A}_x \two &= \rho\two_{x,-1}\mathbf{I}+\rho\two_{x,1}\mathbf{P}_+ \;;&  \lambda_{x}\two(\phi)
&= \rho\two _{x,-1} +\rho\two _{x,1}e^{i \phi}\;. \\
\mathbf{A}_v \one &= \rho\one_{v,1}\mathbf{I}+\rho\one_{v,-1}\mathbf{P}_- \;;&  \lambda_{v}\one(\phi)
&= \rho\one _{v,1} +\rho\one_{v,-1}e^{-i \phi}\;. \\
\mathbf{A}_v \two &= \rho\two_{v,-1}\mathbf{I}+\rho\two_{v,1}\mathbf{P}_+ \;;&  \lambda_{v}\two(\phi)
&= \rho\two _{v,-1} +\rho\two _{v,1}e^{i \phi}\;. \\
\mathbf{B}_x \one &= \mathbf{I} + \rho\one_{x,-2}\mathbf{P}_- + \rho\one_{x,2}\mathbf{P}_+ \;;&
\mu_{x}\one(\phi) &= 1+ \rho\one _{x,2}e^{i \phi} +\rho\one_{x,-2}e^{-i \phi}\;. \\
\mathbf{B}_x \two &= \mathbf{I} + \rho\two_{x,-2}\mathbf{P}_- + \rho\two_{x,2}\mathbf{P}_+ \;;&
\mu_{x}\two(\phi) &= 1+ \rho\two _{x,2}e^{i \phi} +\rho\two_{x,-2}e^{-i \phi} \;. \\
\mathbf{B}_v \one &= \mathbf{I} + \rho\one_{v,-2}\mathbf{P}_- + \rho\one_{v,2}\mathbf{P}_+ \;;&
\mu_{v}\one(\phi) &= 1+ \rho\one _{v,2}e^{i \phi} +\rho\one_{v,-2}e^{-i \phi}\;.  \\
\mathbf{B}_v \two &= \mathbf{I} + \rho\two_{v,-2}\mathbf{P}_- + \rho\two_{v,2}\mathbf{P}_+ \;;&
\mu_{v}\two(\phi) &= 1+ \rho\two _{v,2}e^{i \phi} +\rho\two_{v,-2}e^{-i \phi}\;.
\end{array}
\label{matrices_NNN}
\end{equation}

The following proposition is derived in the same way as the analogous proposition in the previous Section.
\begin{prop} The eigenvalues $\nu$ and associated eigenvectors ${\bf u}_\nu(\phi_m)$ of $M$ (with periodic boundary conditions) \color{black} satisfy ${\bf u}_\nu(\phi_m)=\begin{pmatrix}\e_1\mathbf{v}_m, \e_2\mathbf{v}_m, \nu \e_1\mathbf{v}_m, \nu \e_2\mathbf{v}_m \end{pmatrix}^T\;.$
For each $m\in\{0,\cdots n-1\}$ given, there are four eigenpairs (counting multiplicity)
determined by solving the following equation for $\nu$ and $\epsilon_i$ (we dropped the argument $\phi$):
\begin{align*}
\begin{pmatrix}
g_x\one\mu\one_{x} + \nu g_v\one\mu\one_{v} - \nu^2 & g_x\one \lambda_{x}\one + \nu g_v\one \lambda_{v}\one \\
g_x\two \lambda_{x}\two + \nu g_v\two \lambda_{v}\two & g_x\two\mu\two_{x} + \nu g_v\two\mu\two_{v}-\nu^2
\end{pmatrix}
\begin{pmatrix}
\epsilon_1 \\ \epsilon_2
\end{pmatrix}
&=
\begin{pmatrix}
0 \\ 0
\end{pmatrix}
\;.
\end{align*}
\label{prop:charpoly2}
\end{prop}

\begin{lem} When $\phi=0$, the matrix of Proposition \ref{prop:charpoly2} has determinant
$Q(\nu,\phi=0)$ equal to
\begin{equation*}
\nu^2\,\left[\nu^2+\nu \left(g_v^{(1)}\alpha_{v,1}^{(1)}+g_v^{(2)}\alpha_{v,1}^{(2)}\right)+
 \left(g_x^{(1)}\alpha_{x,1}^{(1)}+g_x^{(2)}\alpha_{x,1}^{(2)}\right)\right]\;.
\end{equation*}
The expression of the constant term of $Q(\nu,\phi)$ is $a_0(\phi)$, where $a_0(\phi)=g\one_x g\two_x \left(\mu\one_{x}(\phi)  \mu\two_{x}(\phi) - \lambda\one_x (\phi)
\lambda\two_{x}(\phi)\right)$.
\label{lem:bookkeeping2}
\end{lem}

\begin{proof}
The full determinant of the matrix in Proposition \ref{prop:charpoly2} is equal to
\begin{equation*}
\begin{array}{rl}
& g\one_x g\two_x \left(\mu\one_{x}  \mu\two_{x} - \lambda\one_x \lambda\two_{x}\right) + \nu \left(g\one_x g\two_v \left(\mu\one_{x} \mu\two _{v} - \lambda\one_x\lambda\two_{v}\right)
+g\one_v g\two_x \left( \mu\one_{v} \mu\two_{x} - \lambda\one_v\lambda\two_{x}\right)\right) \\
+\nu^2& \left(-g\one_x\mu\one_{x}-g\two_x\mu\two_{x}+g\one_v g\two_v\left(\mu\one_{v} \mu\two_{v}- \lambda\two_v\lambda\two_{v}\right) \right)+\nu^3 \left(-g\one_v\mu\one_{v}-g\two_v\mu\two_{v}\right)+\nu^4 
\label{fourth_degree_general}
\end{array} \quad .
\end{equation*}
Now set $\phi=0$. From \eqref{matrices_NNN} and recalling Definition \ref{def:alphas-betas}, we see
that for $r\in\{x,v\}$ and $i\in\{1,2\}$:
\begin{equation*}
\mu_r^{(i)}(0)=1+\alpha_{r,2}^{(i)} \quad \textrm{and} \quad
 \lambda_r^{(i)}(0)=\alpha_{r,1}^{(i)}\;.
\end{equation*}
Note that the constraint \eqref{eq:constraint_minus_one} gives for $r\in\{x,v\}$, $ 1+\alpha_{r,1}^{(i)}+\alpha_{r,2}^{(i)}=0\quad \Longrightarrow \quad -\mu_r^{(i)}(0)=\lambda_r^{(i)}(0)=\alpha_{r,1}^{(i)}$
Substituting this, and some algebra, yields the Lemma.
\end{proof}

\begin{thm} If any of the following conditions are satisfied, then for large $N$, the system given by
\eqref{NN_diatomic_sys} with periodic boundary conditions is not (linearly) stable.

(i) $g_x\one\neq 0$  $g_x\two\neq 0$.

(ii) $g_x^{(1)}\alpha_{x,1}^{(1)}+g_x^{(2)}\alpha_{x,1}^{(2)}\leq 0$ or
$g_v^{(1)}\alpha_{v,1}^{(1)}+g_v^{(2)}\alpha_{v,1}^{(2)}\leq 0$.

(iii) $g\one_x g\two_x \left[\alpha\two_{x,1}\left(\beta\one_{x,1} + 2\beta\one_{x,2}\right) +\alpha\one_{x,1}\left(\beta\two_{x,1} + 2\beta\two_{x,2}\right)\right] \neq 0 $.
\label{NN_diatomic_theorem}
\end{thm}
\color{black}

\begin{proof} This proof is very similar to that of Theorem
\ref{theorem_triatomic}. Part (ii) is now more easily derived by explicitly solving for
the roots of $Q(\nu,\phi)$ when $\phi=0$ (see Lemma \ref{lem:bookkeeping2}).
In (iii), it is best to differentiate the formula
in the second part of Lemma \ref{lem:bookkeeping2} directly. The derivatives of
the $\lambda$'s and $\mu$'s are easily expressed directly in the $\alpha$'s and $\beta$'s.
\end{proof}

The conjectures of \cite{CantosTrans} state that instability in the system
with periodic boundary imply instability or flock instability of the system on the line
(i.e. with non-trivial boundary conditions). That gives us the following corollary.
\color{black}

\begin{cor} Using the conjectures of \cite{CantosTrans}, we obtain that if 
\begin{align*}
\left[-g\one_x\mu\one_{x}-g\two_x\mu\two_{x}+g\one_v g\two_v(\mu\one_{v} \mu\two_{v}- \lambda\two_v\lambda\two_{v})\right]_{\phi=0}\neq 0
\end{align*}
 and
\begin{equation*}
 \displaystyle g\one_x g\two_x \,\left[\alpha\two_{x,1}\left(\beta\one_{x,1} + 2\beta\one_{x,2}\right) + \alpha\one_{x,1}\left(\beta\two_{x,1} + 2\beta\two_{x,2}\right)\right] \ne 0\;,
\end{equation*}
\noindent
then the system on the line given by (\ref{NN_diatomic_sys})
has some form of instability (Definitions \ref{def_stability} or \ref{def_flock-stability}).
\label{cor:diatomic}
\end{cor}

\noindent
As in Corollary \ref{cor:triatomic}, \color{black} the first condition, which corresponds to $a_2(0)\neq 0$ in Proposition \ref{prop:appendix},
can be omitted.
\color{black}

\begin{figure}[pbth]
\begin{center}
\begin{subfigure}{.32\textwidth}
\centering
\includegraphics[width=6.5cm,height=7.5cm]{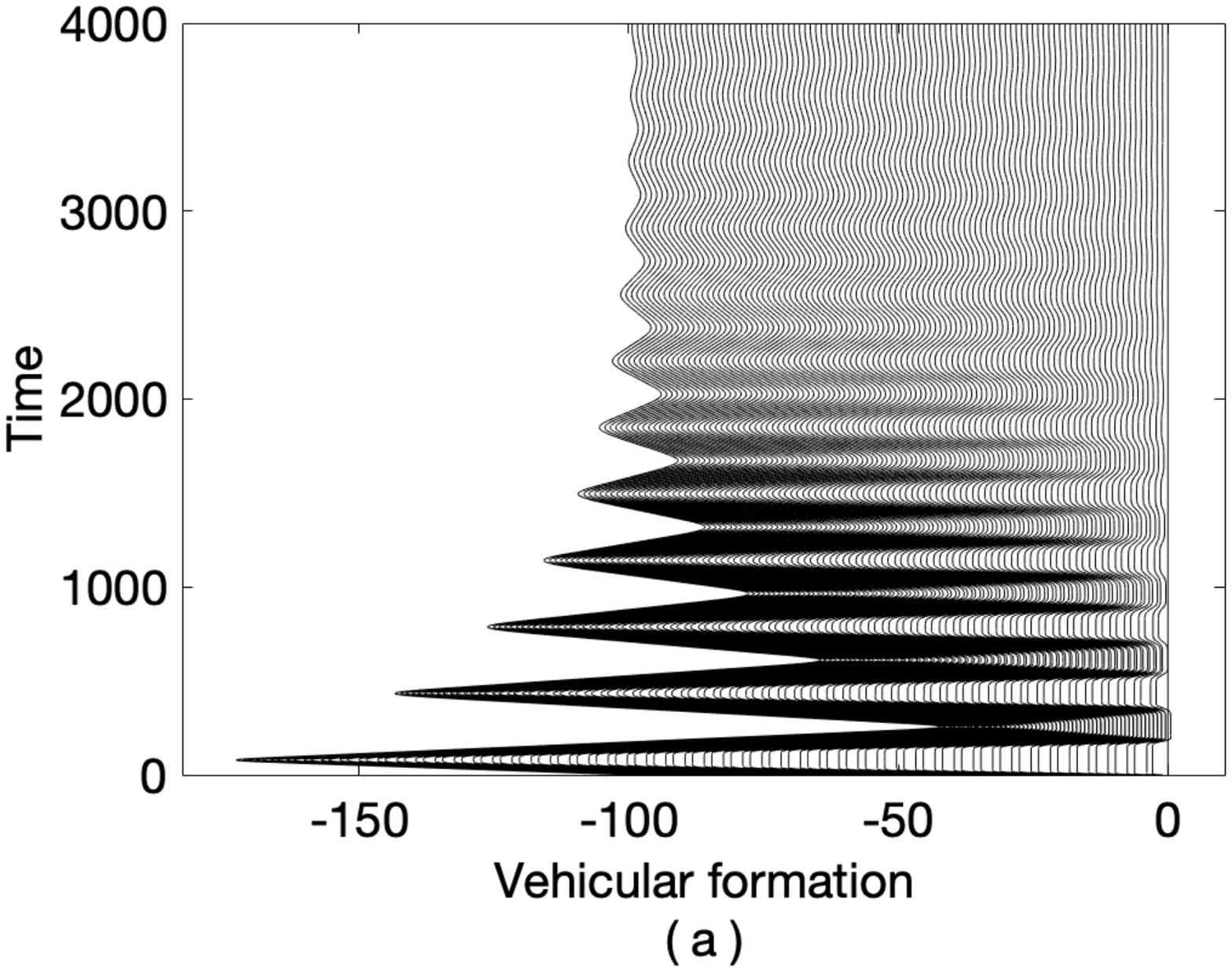}
\end{subfigure}%
\begin{subfigure}{.32\textwidth}
\centering
\includegraphics[width=6.5cm,height=7.5cm]{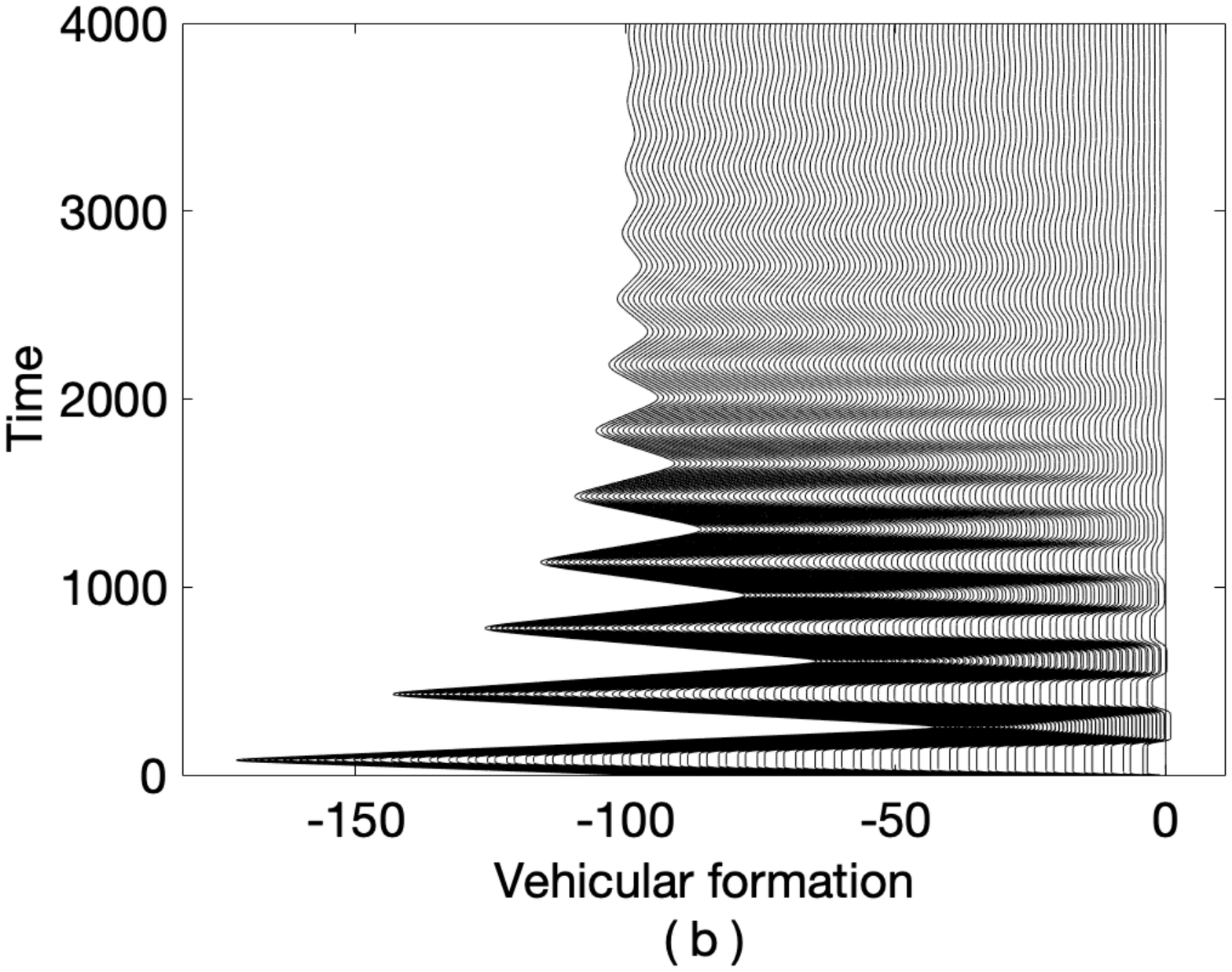}
\end{subfigure}
\begin{subfigure}{.32\textwidth}
\centering
\includegraphics[width=6.5cm,height=7.5cm]{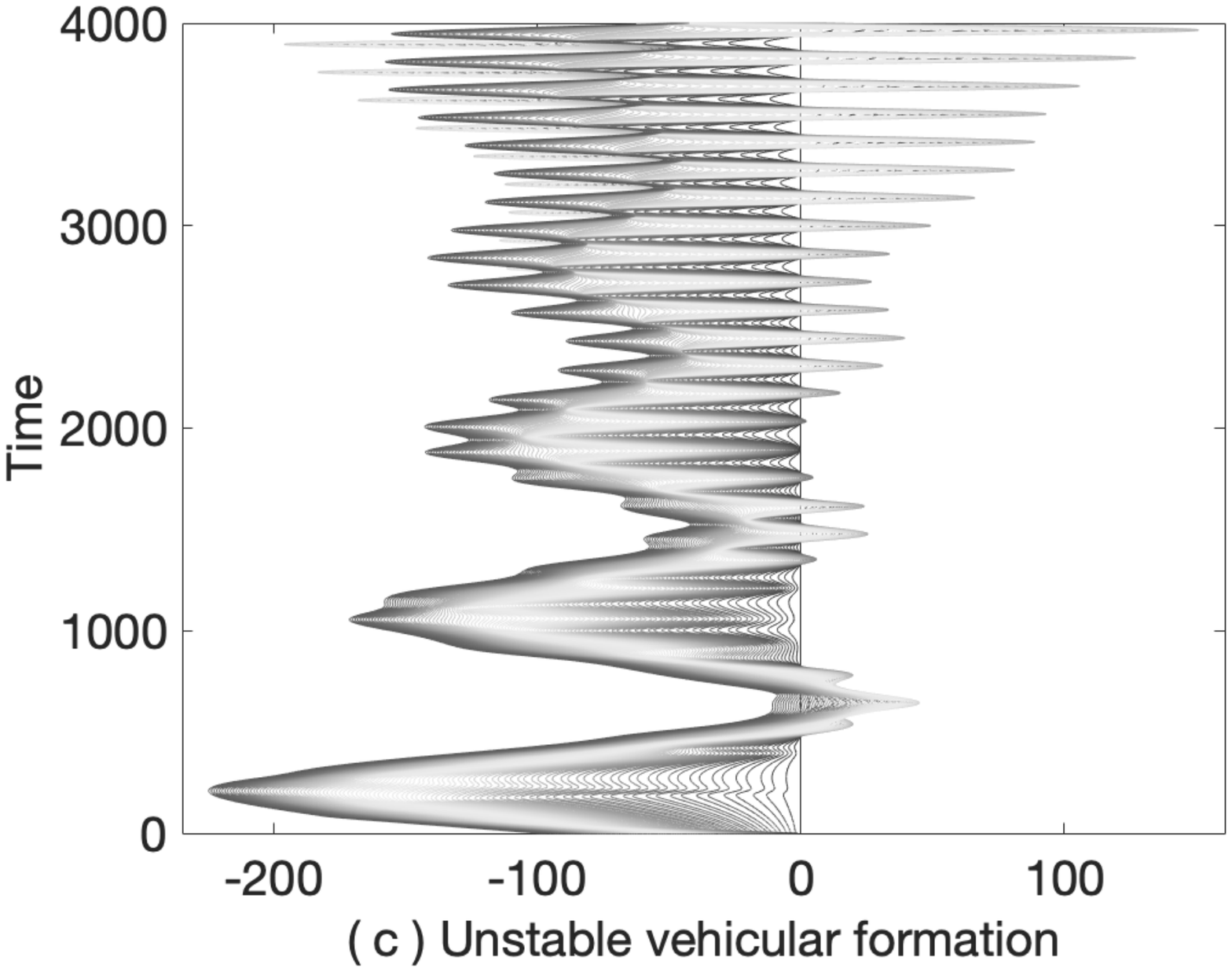}
\end{subfigure}
\vspace{-3\baselineskip}
\caption{\emph{Behavior of vehicular formations with $N=100$, $\boldsymbol{g} = -(1,1,1,1)$, $\boldsymbol{\rho_x} = -\frac{1}{60}( 5,15,20,20,27,9,12,12)$, and $\boldsymbol{\rho_v} = -(0.30, 0.70, 0, 0, 0.30, 0.70, 0, 0)$.} (a) \emph{Boundary Condition Type I. Maximum amplitude of $-72.8$ at $t=79.3$.} (b) \emph{Boundary Condition Type II. Maximum amplitude of $-72.0$ at $t=78.5$.} (c) \emph{Dynamics of an unstable system. All parameters are the same as in (a) \emph{except}  $\boldsymbol{\rho_x} = -(0.30,0.25,0.25,0.20, 0.30, 0.55, 0.10)$.}}
\label{fig:picNNN2}
\end{center}
\end{figure}

In order to do the simulations, we define two types of boundary conditions. \color{black}
In Type I boundary
conditions, the central coefficients $\rho^{(i)}_{x,0} $, and $\rho^{(i)}_{v,0}$ are adjusted.
\begin{align*}
\begin{split}
\ddot{z}\one_1 & =0\\
\ddot{z}\one_n & = g\one_x\left(-(\rho\one_{x,1}+\rho\one_{x,-1}+\rho\one_{x,-2})z\one_n+\rho\one_{x,1}z\two_n+\rho\one_{x,-1}z\two_{n-1}+\rho\one_{x,-2}z\one_{n-1}\right)\\
& ~~~+ g\one_v\left(-(\rho\one_{v,1}+\rho\one_{v,-1}+\rho\one_{v,-2})\dot{z}\one_n+\rho\one_{v,1}\dot{z}\two_n+\rho\one_{v,-1}\dot{z}\two_{n-1}+\rho\one_{v,-2}\dot{z}\one_{n-1}\right)\\
\ddot{z}\two_1 & = g\two_x\left(-(\rho\two_{x,-1}+\rho\two_{x,1}+\rho\two_{x,2})z\two_1+\rho\two_{x,-1}z\one_1+\rho\two_{x,1}z\one_{2}+\rho\two_{x,2}z\two_{2}\right) \\
&~~~ +g\two_v\left(-(\rho\two_{v,-1}+\rho\two_{v,1}+\rho\two_{v,2})\dot{z}\two_1+\rho\two_{v,-1}\dot{z}\one_1+\rho\two_{v,1}\dot{z}\one_{2}+\rho\two_{v,2}\dot{z}\two_{2}\right)\\
\ddot{z}\two_n & = g\two_x\left(-(\rho\two_{x,-1}+\rho\two_{x,-2})z\two_n+\rho\two_{x,-1}z\one_n+\rho\two_{x,-2}z\two_{n-1}\right) \\
&~~~ +g\two_v\left(-(\rho\two_{v,-1}+\rho\two_{v,-2})\dot{z}\two_n+\rho\two_{v,-1}\dot{z}\one_n+\rho\two_{v,-2}\dot{z}\two_{n-1}\right)
\end{split} \quad .
\end{align*}

For Type II BC, we keep the central coefficients $\rho^{(i)}_{x,0} $, and $\rho^{(i)}_{v,0}$ equal to 1 and we adjust the remaining coefficients accordingly such that the sum of coefficients is zero as follows:
\begin{align*}
\begin{split}
\ddot{z}\one_1 & =0\\
\ddot{z}\one_n & = g\one_x\left(z\one_n+\rho\one_{x,1}z\two_n+\rho\one_{x,-1}z\two_{n-1}-(1+\rho\one_{x,1}+\rho\one_{x,-1})z\one_{n-1}\right)\\
& ~~~+ g\one_v\left(\dot{z}\one_n+\rho\one_{v,1}\dot{z}\two_n+\rho\one_{v,-1}\dot{z}\two_{n-1}-(1+\rho\one_{z,1}+\rho\one_{v,-1})\dot{z}\one_{n-1}\right) \\
\ddot{z}\two_1 & = g\two_x\left(z\two_1+\rho\two_{x,-1}z\one_1+\rho\two_{x,1}z\one_{2}-(1+\rho\two_{x,1}+\rho\two_{x,-1})z\two_{2}\right) \\
&~~~ +g\two_v\left(\dot{z}\two_1+\rho\two_{v,-1}\dot{z}\one_1+\rho\two_{v,1}\dot{z}\one_{2}-(1+\rho\two_{v,1}+\rho\two_{v,-1})\dot{z}\two_{2}\right) \\
\ddot{z}\two_n & = g\two_x\left(z\two_n+(\rho\two_{x,1}+\rho\two_{x,-1})z\one_n+(\rho\two_{x,2}+\rho\two_{x,-2})z\two_{n-1}\right) \\
&~~~ +g\two_v\left(\dot{z}\two_n+(\rho\two_{v,1}+\rho\two_{v,-1})\dot{z}\one_n+(\rho\two_{v,2}+\rho\two_{v,-2})\dot{z}\two_{n-1}\right)
\end{split}
\end{align*}

We run simulations of the system in $\mathbb{R}$ considering these two boundary conditions with
initial condition:
\begin{align*}
z^{(i)}_k(0) = \dot z^{(i)}_k(0) = 0 \quad \textrm{\underline{except}} \quad
\dot{z}^{(1)}_1(0) = 1\;.
\end{align*}

As in the previous chapter, we can shorten the notation by writing the coefficients of the system as $\boldsymbol{g} = (g\one_x, g\two_x, g\two_x, g\two_v)$, $\boldsymbol{\rho_x} = (\rho\one_{x,1},\rho\one_{x,-1},\rho\one_{x,2},\rho\one_{x,-2},\rho\two_{x,1},\rho\two_{x,-1},\rho\two_{x,2},\rho\two_{x,-2}) $, and similarly $\boldsymbol{\rho_v}$.
Figures \ref{fig:picNNN2}(a) and \ref{fig:picNNN2}(b) show the dynamics of a system of 100 vehicles in formation with next nearest neighbor interactions and boundary conditions Type I and Type II respectively. The parameters were chosen to satisfy Theorem \ref{NN_diatomic_theorem}.
On the other hand, Figure \ref{fig:picNNN2}(c) shows the dynamics of a system
in $\mathbb{R}$ with an evident instability of some type, see parameters in the caption of Figure \ref{fig:picNNN2}(c). 
These were chosen to satisfy $\sum_{i\in\{1,2\}}\,\beta_{x,1}^{(i)} + 2\beta_{x,2}^{(i)}=0$,
but not the condition of Corollary \ref{cor:diatomic}.

\section{Conclusion}

We consider systems of the form (\ref{eq:laplacian-system}) where a) the Laplacians $L_x$ and $L_v$ do not
necessarily commute and so cannot be simultaneously diagonalized and b) these Laplacians are not
necessarily symmetric. Such systems \emph{cannot} successfully be analyzed by methods used in earlier papers:
Laplace or Fourier transforms, analysis of the eigenvalues and eigenvectors of either ${\bf M}$ or
its constituent Laplacians. Instead, we follow the analysis proposed in \cite{cantos2016signal, CantosTrans}
to analyze these systems.

Our aim with this work is to find conditions for stability for systems in which the agents are \emph{not}
identical. Because this is analytically a very difficult problem, we start in this paper with
periodic arrangements of 3 types of agents ($\cdots 3-2-1-3-2-1$) with nearest neighbor interactions and
periodic arrangements of 2 types of agents ($\cdots 2-1-2-1$)
with \emph{next} nearest neighbor interactions. For these types of flocks, we develop \emph{necessary}
conditions for stability. Corollaries \ref{cor:diatomic} and \ref{cor:triatomic} show that in each of these
two cases, a necessary condition for stability is that $\sum_{j\neq 0} \rho_{x,j}j$
\emph{plus a nonlinear correction}. Thus \emph{stability is a co-dimension one phenomenon}.

We close with a few remarks about these results. The first is that in this context, instability
refers to two phenomena. One is instability in the usual sense of the word (as spelled out in
Definition \ref{def_stability}), namely an eigenvalue has positive real part. The other notion
of instability is given in Definition \ref{def_flock-stability}. This notion essentially means
that transients increase exponentially fast as the number of agents $n$ increases, even though
for each $n$ the system is stable in the sense of Definition \ref{def_stability}.

The second remark is that certainly the necessary condition derived here is not sufficient.
For example, if we give the last agent an infinite mass (setting $g_x=g_v=0$ for this agent), it cannot
change its velocity. Clearly, if the leader changes its velocity, a system with that boundary condition
cannot evolve towards equilibrium.

\color{black}

\section{Appendix}

\noindent
\begin{prop} For $n\geq 2$, define $Q_n(z)=\sum_{i=2}^n\,a_i(t)z^i +2a_1(t)z+a_0(t)$
where the $a_i$ are analytic functions on $\mathbb{R}$ modulo $2\pi$ into $\mathbb{C}$. Assume further that $~a_0(0)=a_1(0)=0 \quad \textrm{and} \quad a_2(0)\neq 0 \quad \textrm{and} \quad a_0'(0)\neq 0$.
Then there is a neighborhood $N$ of the origin and an $\epsilon>0$ in which the zeros of
$\{Q_n(t)\}_{t\in (-\epsilon,\epsilon)}$ form two differentiable curves intersecting orthogonally
at the origin.
\label{prop:appendix}
\end{prop}

\noindent
In particular, it follows that near the origin, the solutions form a perpendicular cross
and thus at least one on the arms of the cross extends into the right half-plane.

\begin{proof}
We start with $n=2$. In this case, we can write out the solutions:
\begin{equation*}
z_\pm(t)=\frac{-a_1\pm \sqrt{-a_0a_2+a_1^2}}{a_2}=
\pm \sqrt{-{a_0}/{a_2}}\,\,\sqrt{1-{a_1^2}/\left({a_0a_2}\right)}\,\,\,-\frac{a_1}{a_2}\;.
\end{equation*}
Let us define a curve $\delta(t)$ to be tangent to a curve $\eta(t)$ at the origin for $t=0$
if $\delta(0)=\eta(0)=0$ and
\begin{equation}
\lim_{t\searrow 0} \dfrac{|\delta(t)-\eta(t)|}{|\eta(t)|}=0\;.
\label{eq:tangent}
\end{equation}
One checks that we need all the assumptions on the coefficients $a_i$, $i\in\{0,1,2\}$, to show that
$z_\pm(t)$ is tangent to $\pm \sqrt{-{a_0'(0)}/{a_2(0)}\,\,t}$. We proceed by doing $n-2$ induction steps. Given $Q_n$, we form all the intermediate polynomials
$\{Q_k\}_{k=2}^n$. Consider $t\in N_\epsilon=(-\epsilon,\epsilon)$ for $\epsilon$ small.
We wish to prove that $t\in N_\epsilon$, the solutions of $Q_k$ form two curves $z_{k,\pm}(t)$
tangent (in the sense of equation \ref{eq:tangent}) at the origin to $\pm \sqrt{-{a_0'(0)}/{a_2(0)}\,\,t}$ which we will from now one denote by $\pm \sqrt{ct}$. See Figure \ref{fig-induction-argument}.

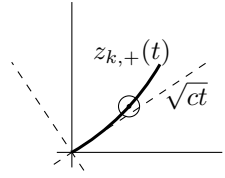
\begin{figure}[h]
\centering
\begin{tikzpicture}[scale = .4]
\draw (-.5,0)--(5,0);
\draw (0,-.5)--(0,5);
\draw[black, line width = 0.4mm] (0.,0) arc (-60:  -30:8cm);

\draw[dashed] (-.6,-.4)--(4.5,3);
\draw[dashed] (.4,-.6)--(-2,3);

\draw (1.9,1.53) circle [radius = .35];
\draw[fill = black] (1.9,1.53) circle [radius = .06];

\draw (2,3.3) node {$z_{k,+}(t)$};
\draw (3.8,1.9) node {$\sqrt{ct}$};
\end{tikzpicture}
\caption{\emph{The curve $\gamma_L$ around $z_{k,+}(t)$ (solid) which itself is on a curve tangent to $\sqrt{ct}$ (dashed).}}
\label{fig-induction-argument}
\end{figure}

We proved the statement holds for $n=2$. The induction hypothesis is that the above statement
holds for some fixed $k\in \{2,\cdots n-1\}$.
Fix an arbitrarily large $L$, (and at least as large as $n$)..
Then fix $\epsilon>0$ small enough, so that the conditions
in the following hold for all $t\in N_\epsilon$. Without loss of generality, take $t\geq 0$ and specialize to one branch, namely $z_{k,+}(t)$. $Q_k$ has no other zeros in an $2\sqrt{|c\epsilon|}$ neighborhood of the origin.
By continuity, for $|z|<\sqrt{|c\epsilon|}$, we can write $Q_k$ as
$(z-z_{k,+})(z-z_{k,-})\tilde Q_k(t,z)$, where $|\tilde Q_k(t,z)|\geq \frac 12|\tilde Q_k(0,0)|\ne 0$.
Similarly, we may assume that $|a_{k+1}(t)|\leq 2|a_{k+1}(0)|$.
Let $\gamma_L(s)$ be the curve $z_{k,+}(t)+\frac{|z_{k,+}(t)|}{L}\,e^{is}$. Then $\gamma_L$ contains no zeroes.
By the induction hypothesis, $z_{k,+}(t)$ is tangent to $\sqrt{ct}$ or:
\begin{equation*}
|z_{k,+}(t)-\sqrt{|ct|}|\leq \frac{1}{L} \sqrt{|ct|} \quad \Longleftrightarrow \quad
\left(1-L^{-1}\right)\sqrt{|ct|}\leq |z_{k,+}(t)| \leq \left(1+L^{-1}\right)\sqrt{|ct|}\;.
\end{equation*}
\begin{equation*}
\begin{array}{ccl}
|a_{k+1}(t)\gamma_L^{k+1}| &\leq & |a_{k+1}(t)|\; |z_{k,+}(t)|^{k+1}\;|1+L^{-1}|^{k+1}\\[0.2cm]
                    & \leq & 2|a_{k+1}(0)|\;|1+\frac{1}{k+1}|^{k+1}|ct|^\frac{k+1}{2}\, |1+\frac{1}{k+1}|^{k+1} = 2e^2 \,|a_{k+1}(0)|\;|ct|^\frac{k+1}{2}\;.\\[0.2cm]
                    \end{array}
                    \end{equation*}
                    \begin{equation*}
                    \begin{array}{ccl}
|Q_k(\gamma_L)|  & = & |\gamma_L-z_{k,+}|\;|\gamma_L-z_{k,-}|\;|\tilde Q_k(t,\gamma_L)| \quad \textrm{where} \quad
         \tilde Q_k(0,0)\neq 0\\[0.2cm]
 &= & \frac{|z_{k,+}(t)|}{L} \;|z_{k,+}(t)+\frac{|z_{k,+}(t)|}{L}\,e^{is}-z_{k,-}(t)|
\;|\tilde Q_k(t,z)|\\[0.2cm]
           &\geq & (L^{-1}-L^{-2})\sqrt{|ct|}\; \dfrac{\sqrt{|ct|}}{2} \;
            \dfrac{|\tilde Q_k(0,0)|}{2}\;.\end{array}
\end{equation*}
Thus we can choose $t$ small enough so that, on $\gamma_L$, $|a_{k+1}(t)z^{k+1}|$ is smaller
than $|Q_k(z)|$. Since neither
function has poles, Rouch\'{e}'s theorem \cite{marsden1999basic} implies that $a_{k+1}(t)z^{k+1}+Q_k(z)$
has the same number of zeros inside $\gamma_L$ as does $Q_k(z)$, namely one.
Thus $Q_{k+1}(z)$ has a unique zero within $\gamma_L$. Since we can do this for any value
of $L$ (at the price of making $\epsilon$ small enough), it follows that $z_{k+1,+}(t)$
is tangent to $z_{k,+}(t)$ and hence to $\sqrt{ct}$. Since we need only finitely
many induction steps to get to $z_{n,+}(t)$, the statement of the proposition follows.
\end{proof}


\singlespacing      
\bibliographystyle{plain}
\bibliography{references}

\end{document}